\pdfoutput=1
\documentclass[11pt]{amsart}
\usepackage{amsmath}
\usepackage[utf8]{inputenc}
\usepackage{amssymb}
\usepackage{harpoon}
\usepackage[pdftex]{graphicx}
\usepackage{hyperref}
\usepackage{esint}
\usepackage{relsize}
\usepackage{xcolor}

\definecolor{brass}{rgb}{0.71, 0.65, 0.36}

\usepackage{tikz-cd}
\theoremstyle{plain}
\newtheorem{theorem}{Theorem}[section]
\newtheorem{lemma}[theorem]{Lemma}

\theoremstyle{definition}
\newtheorem{remark}[theorem]{Remark}

\numberwithin{equation}{section}

\def\Ric{\operatorname{Ric}}

\def\R{\mathbb{R}}

\def\sup{\operatorname{sup}}
\def\inf{\operatorname{inf}}

\def\a{\alpha}
\def\p{\partial}
\def\e{\varepsilon}

\def\n{\nabla}
\def\oo{\infty}
\def\sup{\operatorname{sup}}
\def\supp{\operatorname{supp}}

\def\div{\operatorname{div}}

\def\[{\left[}
\def\]{\right]}
\def\({\left(}
\def\){\right)}
\def\<{\langle}
\def\>{\rangle}
\theoremstyle{plain}
\def\vp{\varphi}

\numberwithin{equation}{section}

\allowdisplaybreaks

\def\dvolg{\mathrm{dVol_g}}

 \def\O{\Omega}

  \def\tab{\;\;\;\;\;\;}
  \def\dmu{e^{-f}\dvolg}

\def\np{\n_{\p^*\Omega}}

\begin{document}
\title{A Spectral Splitting Theorem for the $N$-Bakry \'Emery Ricci tensor}

\author{Wai-Ho Yeung}
\address{Department of Mathematics\\
The Hong Kong University of Science and Technology}
\email{whyeungae@connect.ust.hk}

\begin{abstract}
We extend the spectral generalization of the Cheeger-Gromoll splitting theorem to smooth metric measure space. We show that if a complete non-compact weighted Riemannian manifold $(M,g,e^{-f}\,dvolg)$ of dimension $n\ge 2$ has at least two ends where $f$ is smooth and bounded. If there is some $N\in (0,\infty)$ and $\gamma<\left(\frac{1}{(n-1)\left(1 + \frac{n-1}{N}\right)} + \frac{n-1}{4}\right)^{-1}$ such that $$\lambda_1(-\gamma \Delta_f+\operatorname{Ric}^N_f)\ge 0$$  then $M$ splits isometrically as $\mathbb{R}\times X$ for some complete Riemannian manifold $X$ with $(\operatorname{Ric}_X)^N_f\ge 0$. The estimate can recover the spectral splitting result and its sharp constant $\frac{4}{n-1}$ in \cite{spectral_splitting} and \cite{spectral_splitting_2}.
\end{abstract}

\maketitle

\section{Introduction}
The celebrated Cheeger-Gromoll splitting theorem \cite{CheegerSplitting} asserts that a complete Riemannian manifold $(M,g)$ with nonnegative Ricci curvature and containing a geodesic line must split isometrically. Consequently, a complete manifold with nonnegative Ricci curvature and at least two ends must split.

Following Bakry \'Emery \cite{bakrydiff}\cite{bakryhyper}\cite{bakry2005volume}, given a smooth function $f\in C^\infty(M)$ and induce the smooth metric measure space $(M,g,e^{-f}\dvolg)$. The classical Ricci tensor has been generalized to $N$-dimensional Bakry-\'Emery Ricci curvature as follows $$\Ric^N_f:=\Ric+\mathrm{Hess}(f)-\dfrac{1}{N}df\otimes df$$ 

The comparison theory and the structure theory for the Bakry-\'Emery Ricci tensor lower bound with bounded $f$ has been studied extensively in the past. Most of the comparison theorems for the classical Ricci tensor have also been discovered to be true in the Bakry-\'Emery Ricci tensor, such as volume comparison, mean curvature comparison, splitting theorem, etc. See for example, \cite{lichn1}\cite{lichn2}\cite{wei2akryemeryricci}\cite{fang}.

In \cite{AX24}, Antonelli-Xu proved a sharp volume comparison and Bonnet-Myers theorem for manifolds with positive Ricci curvature in the spectral sense. More recently, Antonelli-Pozzetta-Xu \cite{spectral_splitting} and Catino--Mari--Mastrolia--Roncoroni\cite{spectral_splitting_2} generalized the Cheeger Gromoll splitting theorem in the spectral sense independently. In \cite{spectral_splitting}, following the $\mu$-bubble method from Gromov \cite{gromov4lectures}, they proved the following result.\begin{theorem}\label{thm:APX-result}
    Let $n\geq 2$, $\gamma<\frac{4}{n-1}$, and let $(M^n,g)$ be an $n$-dimensional smooth complete noncompact Riemannian manifold without boundary. Assume that $M$ has at least two ends and satisfies
    \begin{equation}\label{eq:main_condition}
        \lambda_1(-\gamma\Delta+\Ric)\geq0.
    \end{equation}Here $\Ric:M\to\R$ by $$\mathrm{Ric}(x):=\inf_{\substack{v\in T_xM\\ g(v,v)=1}} \mathrm{Ric}_x(v,v).$$
    Then $\Ric\geq0$ on $M$. In particular, the manifold splits isometrically as $(M,g)\cong(\R\times N,\mathrm{d}t^2+g_N)$ for some compact manifold $N$ with nonnegative Ricci curvature.
\end{theorem}
Motivated by Theorem \ref{thm:APX-result} and spectral comparison results for N-Bakry-\'Emery Ricci tensor by Chu-Hao\cite{Chu_BE}, we aim to generalize the result for the N-Bakry-Emery Ricci tensor in the spectral sense by adapting the approach of Antonelli-Porzetta-Xu \cite{spectral_splitting}.\begin{theorem}\label{thm:main_thm}
    Let $(M,g,e^{-f}\dvolg)$ be a smooth complete $n$-dimensional noncompact metric measure space without boundary with $n\ge 2$ and $N,\gamma$ be two constants such that $$N\in(0,+\infty),\tab \gamma< \left(\frac{1}{(n-1)\left(1 + \frac{n-1}{N}\right)} + \frac{n-1}{4}\right)^{-1}$$Furthermore, assuming $M$ has at least two ends and for some function $f\in C^\infty(M)$ such that \begin{equation}\label{eq:main condition}
        \lambda_1(-\gamma\Delta_f+\Ric^N_f)\ge 0
    \end{equation}Here $$\Ric^N_f(x):=\inf_{\substack{v\in T_xM\\ g(v,v)=1}} (\Ric^N_f)_x(v,v) $$denotes the smallest eigenvalue of the $N$-Bakry \'Emery Ricci tensor at $x$. Assume further that $||f||_{0}<\infty$, then $\Ric^N_f\ge 0$.In particular, $M$ splits isometrically as $\R\times X$ for some complete Riemannian manifold $X$ without line and the Euclidean space $\R$. Furthermore, the function $f$ is constant on each $\R$-factor and $X$ has  nonnegative $N$-Bakry \'Emery Ricci tensor. 
\end{theorem}Note that $\Ric^N_f\in Lip_{loc}(M)$. For a constant $\gamma\ge 0$, we say that $M$ satisfies $\lambda_1(-\gamma\Delta_f+\Ric^N_f)\ge 0$ if any of the following equivalent conditions hold:\begin{itemize}
    \item [(i)] For all $\varphi\in C^1_c(M)$, it holds $\int_M(\gamma|\n\varphi|^2+\Ric^N_f\cdot\varphi^2)e^{-f}\dvolg\geq0$
    \item[(ii)]There exist $\alpha\in (0,1)$ and $u\in C^{2,\alpha}(M)$ such that $u>0$ and $-\gamma\Delta_f u+\Ric^N_f\cdot u\geq0$ on $M$.
\end{itemize}The equivalence of two conditions is readily established on a compact manifold and follows from \cite[section 1]{schoenstable} (see Theorem \ref{thm:equiv+positive}), $(i)$ implies $(ii)$ on non-compact manifolds.

The approach to Theorem \ref{thm:main_thm} is inspired by Antonelli-Pozzetta-Xu \cite{spectral_splitting}. They incorporate the $\mu$-bubbles technique of Gromov\cite{gromov4lectures} and a surface-capturing technique inspired by arguments in \cite{liu20123manifoldsnonnegativericcicurvature}\cite{chu2024kahlermanifoldsnonnegativemixed}\cite{Carlotto_2016}\cite{Chodosh_2018}.

In our setting of Theorem \ref{thm:main_thm}, by Theorem \ref{thm:equiv+positive}, we have a function $u$ with \begin{equation}\label{eq:required equation}
    u>0,\tab -\gamma\Delta_f u+\Ric^N_f\cdot u=0
\end{equation}  

Given that $M$ has at least two ends, an analogue of Zhu's $\mu$-bubble approximation argument produces a hypersurface $\Sigma\subset M$ that locally minimizes the weighted area $\Sigma\mapsto\int_\Sigma u^\gamma e^{-f}$. Since $\Sigma$ might be non-compact in our setting, the foliation argument is no longer available. An approximation scheme for each $x\in M$ gives a weighted minimal hypersurface $\Sigma$ passing through $x$. Then, the stability inequality yields $|\nabla u|=0$ on $\Sigma$. Hence, the arbitrariness of $x$ will imply that $u$ is constant. Thus, $\Ric^N_f\geq0$ and the splitting follows. 
\begin{remark}
The author would like to highlight the differences in methodology between \cite{spectral_splitting} and the estimates presented in \cite{Chu_BE}. Besides the minor difference in rescaling the $\mu$-bubbles by a constant due to the exponential factor, the quantitative estimate in Subsection \ref{subsec: Quantitative Estimate for Quad} is crucial. If directly apply the argument in \cite{Chu_BE}[Page 7] to estimate the normal derivative of $f$, the Young's inequality will lead to an overestimate that $\gamma$ could not recover the sharp constant $\frac{4}{n-1}$ in \cite{spectral_splitting}\cite{spectral_splitting_2}. The quantitative estimate provides a way to deal with $f_\nu$ using the normal derivative of the $\mu$ bubble function $h$ (Check eq(\ref{eq:apply quan estimate})) which leads the possibility to recover the constant.
\end{remark}

\textbf{Organization}. The remaining sections of this paper are structured as follows. We collect some preliminary results in section \ref{sec:Preliminary}. In Section \ref{subsec:f-Lap}, we introduce the $f$-Laplacian, the $N$-Bakry \'Emery Ricci tensor and the Bakry \'Emery Ricci tensor version of splitting theorem. In Section \ref{subsec:spectral}, we modify the spectral result in \cite[Theorem 1]{schoenstable} to fit into our $f$-Laplacian case. In Section \ref{subsec:mu bubble}, we modify the $\mu$-bubble function, the solution perturbation lemma and the density lemma Antonelli-Porzzetta-Xu\cite[section 2]{spectral_splitting}. In Section \ref{subsec: Quantitative Estimate for Quad}, we provide a quantitative estimate for quadratic polynomial to deal with the normal derivative $f_\nu$ in the proof. In Section \ref{sec:main thm}, we prove the main theorem, Theorem \ref{thm:main_thm}. In Section \ref{sec:Alt proof}, we show the Theorem \ref{thm:main_thm} recover the sharp spectral splitting result in the Theorem \ref{thm:APX-result}. 

\textbf{Acknowledgments.} The author would like to thank Gioacchino Antonelli for raising his interest in Ricci's curvature lower bounds in the spectral sense and providing feedback on earlier version of this paper. The author would also like to thank my colleague: Tsz Hong Clive Jr Chan, Chun Szeto and Sung Chit Tai for precious comments on the draft of this paper.
\section{Preliminaries}\label{sec:Preliminary}
In this section, we review some analytical objects, tools and prove some results that will be used throughout the proof of the spectral splitting theorem. We consider $(M,g)$ a Riemannian manifold and $f$ is a smooth function on $M$. It induces a weighted Riemannian manifold triple $(M,g,e^{-f}\dvolg)$, which was introduced by Lichnerowicz \cite{lichn1}\cite{lichn2}.
\subsection{\texorpdfstring{$f$}{}-Laplacian, \texorpdfstring{$N$}{}-Bakry \'Emery Ricci Tensor and splitting theorem }\label{subsec:f-Lap}
The $f$-Laplacian $\Delta_f$ is the differential operator defined by $$\Delta_f u:=e^f\div(e^{-f}\n u)=\Delta u-\<\n u,\n f\>$$It is clear that $\Delta_f$ is a self-adjoint operator with respect to the weighted measure $e^{-f}\dvolg$. Assume $u,v\in C^{2,\alpha}_0(M)$, then we have $$\int_Mu(\Delta_fv)\, e^{-f}\dvolg=-\int_M\<\n u,\n v\>\, e^{-f}\dvolg=\int_Mv(\Delta_fu)\, e^{-f}\dvolg$$Lichnerowicz \cite{lichn1} and Bakry-\'Emery\cite{bakrydiff} introduced a natural generalization of the Ricci tensor$$\Ric_f:=\Ric+\mathrm{Hess}(f)$$More generally, the $N$-Bakry \'Emery Ricci tensor is defined by \begin{equation}\label{def:N-BE-Ricci}
    \Ric^N_f:=\Ric+\mathrm{Hess}(f)-\dfrac{1}{N}df\otimes df
\end{equation}By the definition of the $f$-Laplacian and the classical Bochner formula on Riemannian manifolds, one immediately obtains a Bochner formula for weighted manifolds and the relationship to the $N$-Bakry-\'Emery Ricci tensor.\begin{equation}
    \dfrac{1}{2}\Delta_f|\n u|^2=|\mathrm{Hess}(u)|^2+\<\n u,\n \Delta_f u\>+\Ric_f(\n u, \n u)\ge \dfrac{(\Delta_f u)^2}{N+n}+\<\n u,\n \Delta_f u\>+\Ric^N_f(\n u,\n u)
\end{equation}We recall the classical splitting theorem for the $N$-Bakry \'Emery Ricci tensor proved by Fang,Li and Zhang \cite[Theorem 1.3]{fang}\begin{theorem}\label{thm:splitting theorem BE}
    Let $(M,g)$ be a complete connected Riemannian $n$-dimensional manifold and $f\in C^2(M)$ be a function satisfying that $\Ric^N_f\ge 0$ for some $N>0$, which is not necessary to be an integer. Then $M$ splits isometrically as $X\times \R^\ell$ for some complete Riemannian manifold $X$ without line and the $\ell$-Euclidean space $\R^\ell$. Furthermore, the function $f$ is constant on each $\R^\ell$-factor, and $X$ has a non-negative $(N-\ell)$-dimensional Bakry-\'Emery Ricci curavture.
\end{theorem}
For more comparison geometry results on the Bakry-\'Emery Ricci tensor, we refer to Wei-Wylie's paper \cite{wei2akryemeryricci}.
\subsection{Some general spectral results for the operator \texorpdfstring{$\Delta_f-q$}{}}\label{subsec:spectral}
Let $(M,g)$ be a complete noncompact $n$-dimensional Riemannian manifold and $\dvolg$ be the volume element of $(M,g)$. For $f\in C^\infty(M)$, the triple $(M,g,e^{-f}\dvolg)$ is the associated weighted manifold. Let $q\in C^{\infty}(M)$. Given any bounded domain $D\subset M$, we let $\lambda_1(D)<\lambda_2(D)\le \lambda_3(D)\le \cdots$ be the sequence of eigenvalues of $\Delta_f-q$ that acts on functions that vanish on $\p D$. The usual variational characterization of $\lambda_1(D)$ is \begin{equation}\label{eq:eigenvalue variation}
    \lambda_1(D)=\inf\left\{\int_D(|\n u|^2+qu^2)e^{-f}\dvolg:\mathrm{supp}\, u\subset D,\int_D u^2e^{-f}\dvolg=1\right\}
\end{equation}where $|\n f|^2$ denotes the magnitude of the gradient of $f$ taken with respect to $g$. The following lemma is a consequence of $(\ref{eq:eigenvalue variation})$ and the Unique continuation property of $\Delta_f$\cite{koch}

\begin{lemma}\label{lem: eigenvalue domain comparison}
    If $D,D'$ are connected domains in $ M$ with $D\subset D'$, then $\lambda_1(D)\ge \lambda_1(D')$. If $D'\setminus D\neq \emptyset$, then we have $\lambda_1(D)>\lambda_1(D')$
\end{lemma}

We now state the main result of this section which is an analogue of \cite[Theorem 1]{schoenstable} but replace $\Delta$ with $\Delta_f$ 

\begin{theorem}\label{thm:equiv+positive}
    The following conditions are equivalent:\begin{itemize}
        \item [(i)]$\lambda_1(D)\ge 0$ for every bounded domain $D\subset M$ 
        \item [(ii)]$\lambda_1(D)> 0$ for every bounded domain $D\subset M$ 
        \item[(iii)]There exists a positive function $g$ satisfying the equation $\Delta_f g-qg=0$ on $M$
    \end{itemize}
\end{theorem}
\begin{proof}
\begin{itemize}
    \item [(i)$\to$(ii)]    It is a consequence of Lemma \ref{lem: eigenvalue domain comparison} since for any bounded domain $D\subset M$ and any point $x_0\in M$, we can choose $R$ large enough so that the ball $B_R(x_0)$ contains $D$ such that $B_{R}(x_0)\setminus D\neq \emptyset$. Then we have $\lambda_1(D)>\lambda_1(B_R(x_0))\ge 0$ by assumption for the last inequality.

\item[(ii)$\to$(iii)]    To prove the existence of a positive solution $g$ of the equation $\Delta_f g-qg=,.$ we fix a point $x_0\in M$. For each $R>0$, we consider the problem\begin{equation}\label{eq:positive lemma}
        \begin{cases}
            \Delta u-\left<\n f,\n u\right>-qu&=0\tab \text{on}\tab B_R(x_0)\\
            u&=1\tab \text{on}\tab \p B_R(x_0)\\
        \end{cases}
    \end{equation} Since $\lambda_1(B_R(x_0))>0$, there is no non-zero solution of $\Delta_f u-qu=0$ on $B_R(x_0)$ with $u=0$ on $\p B_R(x_0)$. Since $\Delta_f-q$ is strictly elliptic with coefficient in $C^\alpha(\overline{B_R(x_0)})$ within $B_R(x_0)$. By Fredholm alternative, implies the existence of a unique solution $v\in C^{2,\alpha}(\overline{B_R(x_0)})$ of \begin{equation}
        \begin{cases}
            \Delta_f v-qv&=q\tab \text{on}\tab B_R(x_0)\\
            v&=0\tab \text{on}\tab \p B_R(x_0)\\
        \end{cases}
    \end{equation}It follows that $u=v+1$ is the unqiue solution of $(\ref{eq:positive lemma})$

    After showing the existence, the positivity argument follows exactly the same in \cite[Theorem 1 (ii)$\to$(iii)]{schoenstable}, using strong maximum principle, Harnack inequality and standard elliptic theory to estimate the derivative. Those principles are valid for any weighted manifolds, as the density factor $e^{-f}$ is smooth and bounded. We will continue the proof below.

    To show $u>0$ on $B_R(x_0))$. We show it is non-negative first. Denote $u^-=\min\{u,0\}$ and $\O=\{x:B_R(x_0):u<0\}$. Clearly $u^-=0$ in the $\p^*\Omega$. By \ref{lem: eigenvalue domain comparison} and the assumption, we have $\lambda_1(\O)>0$. It implies $u^-$ is a trivial solution. Since $u\ge 0$ in $B_R(x_0)$. By strong maximum principle, we have $u>0$ in $B_R(x_0)$.

    To extend the solution globally. Take $g_R(x)=u(x_0)^{-1}u(x)$ for $x\in M$. Trivially, it satisfies \begin{equation}
        \begin{cases}
            \Delta_f g_R -qg_R&=0\tab \text{on $B_R(x_0)$}\\
            g_R(x_0)=1,\tab g_R&>0\tab \text{on $B_R(x_0)$}
        \end{cases}
    \end{equation}For any compact set $K$, take $\sigma>0$ such that $K\subset B_\sigma(x_0)$, then for any $R>4\sigma$, $g_R$ are well-defined in $B_{2\sigma}(x_0)$. As $B_{2\sigma}(x_0)$ are compact and $f\in C^{\infty}(M)$, $|\n f|$ is bounded in $B_{2\sigma}(x_0)$. Therefore by Harnack inequality\cite[Corollary 8.21]{GT}, there exists a constant $C$ depends on $B_\sigma(x_0)$ and $B_{2\sigma}(x_0)$ (independent of $R$). Such that for any $z\in B_{\sigma}(x_0)$ \begin{equation}
        g_R(z)\le\sup_{z\in B_{\sigma}(x_0)}g_R(z)\le C  \inf_{z\in B_{\sigma}(x_0)}g_R(z)\le C\tab(g_R(x_0)=1)
    \end{equation}By Schauder estimate \cite[Corollary 6.3]{GT}, the $C^{2,\alpha}$ norm of $g_R$ are also uniformly bounded in $B_\sigma(x_0)$. Therefore, by compact embedding of Holder space \cite[Lemma 6.33]{GT}, there exists a subsequence such that converging to a $C^{2,\beta}$ limit solution $g$ where $0<\beta<\alpha$. Therefore,we have $g_R\to g$ in $C^{2,\beta}_{\mathrm{loc}}(M)$ up to subsequence. We obtain the solution $\Delta_f g-qg=0$ and $g(x_0)=1$ with $g\ge 0$. Follow by strict maximum principle, we have it is a positive solution on $M$.

    \item[(iii)$\to$(i)] If $g>0$ satisfies $\Delta_f g-qg=0$ on $M$, we define a new function $w=\log g$. We now calculate 
    \begin{equation}\label{eq:equiv lemma iii to i}
        \Delta_f w=q-|\n w|
    \end{equation}Let $h$ be any function with compact support on $M$, multiplying $(\ref{eq:equiv lemma iii to i})$ by $h^2$ and integrating by part with respect to $e^{-f}\dvolg$, we obtain \begin{equation}\label{thm: equiv lem iii to i 2}
        -\int_M qh^2 \dmu+\int_M |\n w|^2 h^2\dmu=2\int_M h\<\n h,\n w\>\dmu
    \end{equation}Applying Cauchy-Schwarz inequality and Young's inequality, we have \begin{equation}\label{eq: thm 2 (iii) to (i) 3}
        2|h|\<\n h,\n w\>\le 2|h||\n h||\n w|\le h^2|\n w|^2+|\n h|^2
    \end{equation}Putting $(\ref{eq: thm 2 (iii) to (i) 3})$ to $(\ref{thm: equiv lem iii to i 2})$, and for any $D$ a bounded domain and $f$ is any function with support in $D$, we have \begin{equation}
        0\le \int_M qh^2\dmu+\int_M|\n h|^2\dmu
    \end{equation} By the variational characterization (\ref{eq:eigenvalue variation}), we have $\lambda_1(D)\ge 0$. This finishes the proof of Theorem \ref{thm:equiv+positive}. 
\end{itemize}
\end{proof}
\subsection{Results related to construction of \texorpdfstring{$\mu$}{}-bubbles}\label{subsec:mu bubble}
In this subsection, we collect the result with slight modification in the Preliminary of the spectral splitting paper \cite[section 2]{spectral_splitting}.\begin{lemma}\label{lemma:mu_bubble_function}
    For any $\e\in(0,\frac12)$ and $q>0$ there is a smooth function $\bar h^q_\e:(-\frac q \e,\frac q\e)\to\R$ such that 
    \begin{enumerate}
        \item on $(-\frac q\e,-q]\cup[q,\frac q\e)$ we have
            \begin{equation}
                q(\bar h^q_\e)'+(\bar h^q_\e)^2\geq\e^2,
            \end{equation}
        
        \item on $[-q,q]$ we have
            \begin{equation}
                |q(\bar h^q_\e)'+(\bar h^q_\e)^2|\leq C\e,
            \end{equation}
            for a universal constant $C=C(q)=1+\coth^2(\frac 1 2)+2q^{-1}\coth(\frac 1 2)>0$
        
        \item there holds $(\bar h^q_\e)'<0$ on $(-\frac q\e,\frac q\e)$, $\bar h_\e(0)=0$, $\lim_{x\to\pm\frac q\e} \bar h_\e(x)=\mp\infty$,
        
        \item $\bar h_\e\to0$ smoothly as $\e\to0$ on any compact subset of $\R$.
    \end{enumerate}
\end{lemma}
\begin{proof}
    We can take
    \begin{equation}
        \bar h^q_\e(x):=\e\eta(q^{-1}x)\coth(1+\e q^{-1} x)-\e\big(1-\eta(q^{-1} x)\big)\coth(1-\e q^{-1} x)
    \end{equation}
    where $\eta$ is a smooth cutoff function with $\eta|_{(-\infty,-1]}\equiv1$, $\eta|_{[1,+\infty)}\equiv0$, $-1\leq\eta'\leq0$, and $\eta(0)=1/2$. 
\end{proof}

The next lemma provides the functions we will use to perturb the function $u_0$ satisfying the following\begin{equation}
    u>0,\tab -\gamma\Delta_f u+\Ric^N_f\cdot u=0
\end{equation}.
\begin{lemma}\label{lemma:perturb_new}
    Let $(M,g)$ be a smooth complete noncompact Riemannian $n$-manifold. Let $f\in C^\oo(M)$ with $||f||_{0}<\infty$, and form the weighted manifold $(M,g,e^{-f}\dvolg)$. Let $\gamma>0$, $q\in Lip_{loc}(M)$, and assume there exist $\alpha\in (0,1)$, and $u_0\in C^{2,\alpha}(M)$, such that $-\gamma\Delta_f u_0+qu_0=0$, and $u_0>0$ on $M$. Then for any $x\in M$ and $r\in(0,1)$, there exists $w\in C^{2,\alpha}(M)$ such that
    \begin{enumerate}
        \item $-2u_0\leq w<0$ on $M$,

        \item $-\gamma\Delta_f w+ q w>0$ on $M\setminus B(x,r)$.
    \end{enumerate}
\end{lemma}
\begin{proof}
    The argument just a minor modification of \cite[Lemma 2.2]{spectral_splitting}. Let $U\subset B(x,r)$ be a smooth open set, and let $\eta\in \mathrm{Lip_{loc}}(M)$ be a strictly positive function such that $q+\eta\in C^\infty(M)$, by Theorem \ref{thm:equiv+positive}, the positive solution $u_0$ implies $\lambda_1(-\gamma \Delta_f+q)\ge 0$ on $M$, and therefore $\lambda_1(-\gamma\Delta_f+q+\eta)\ge 0$ on $M\setminus U$ as well. Let $\Omega_1\subset \Omega_2\subset\dots\subset M$ be an exhaustion in smooth pre-compact open set. By Theorem \ref{thm:equiv+positive}, we have $\lambda_1(-\gamma\Delta_f+q+\eta)>0$ on $\Omega_i\setminus U$ for every $i$. Then there exists a solution $w_i\in C^{2,\alpha}(\overline{\Omega}\setminus U)$ to the Dirichlet problem\begin{equation}\label{eq:system 1}
        \begin{cases}
        -\gamma\Delta_f w_i+q w_i&=-\eta w_i\tab\text{in $\Omega_i\setminus \overline{U}$}\\
        w_i|_{\p^*\Omega_i}&=0\\
        w_i|_{\p U}&=u_0
    \end{cases}
    \end{equation}Indeed, we can assume $d(\p U,\p^*\Omega_i)>\delta$ for some $\delta>0$. Then consider a cut off $\xi$ with $\supp \xi\subset T_{\delta}(\p U)$. It is equivalent to consider the following Dirichlet Problem \begin{equation}\label{eq:system 2}
        \begin{cases}
        -\gamma\Delta_f (w_i-\xi u_0)+q (w_i-\xi u_0)+\eta (w_i-\xi u_0)&= 0\tab\text{in $\Omega_i\setminus \overline{U}$}\\
        (w_i-\xi u_0)|_{\p^*\Omega_i}&=0\\
        (w_i-\xi u_0)|_{\p U}&=0
    \end{cases}
    \end{equation}By Fredholm alternative \cite[Theorem 6.15]{GT}, $w_i-\xi u_0$ is a trivial solution. From the system $\ref{eq:system 2}$, we can conclude there exist a solution $w_i$ in the system $(\ref{eq:system 1})$.

    To show $w_i\in \O_i\setminus U$ is a positive function, taking $w^-_i:=\min\{w_i,0\}\in \mathrm{Lip}_0(\O_i\setminus U)$ and notice that its boundary vanish. Therefore, we have $0=\int_{\O_i\setminus U} (\gamma|\n w_i^-|^2+q(w^-_i)^2+\eta(w^-_i)^2)e^{-f}\dvolg$. By $\lambda_1(-\gamma\Delta_f+q+\eta)>0$ and the variational characterization (\ref{eq:eigenvalue variation}), we have $w^-_i=0$. We have $w_i\ge 0$, and $w_i>0$ in $\O_i\setminus U$ by strong maximum principle.

    Considering $w_{i+1}-w_i$, we have  $$\begin{cases}
        -\gamma\Delta_f (w_{i+1}-w_i)^-+q (w_{i+1}-w_i)^-&=-\eta (w_{i+1}-w_i)^-\tab\text{in $\Omega_i\setminus \overline{U}$}\\
        (w_{i+1}-w_i)^-|_{\p^*\Omega_i}&=(w_{i+1})^-=0\\
        (w_{i+1}-w_i)^-|_{\p U}&=0
    \end{cases}$$Similarly, by $\lambda_1(-\gamma\Delta_f+q+\eta)>0$ and the variational characterization (\ref{eq:eigenvalue variation}), we have $(w_{i+1}-w_i)^-=0$. We have $w_{i+1}-w_i\ge 0$ on $\O_i\setminus U$. By strong maximum principle, we have $w_{i+1}>w_i$

    Moreover, considering $-\gamma\Delta_f(u_0-w_i)+q(u_0-w_i)=\eta w_i>0$ in $\O_i\setminus U$ for any $i$. Using similar approach, we have $w_i<u_0$ for any $i$

    Fix $K$ be a compact subset in $M\setminus U$. Since it is compact, there exists $i_0>0$ such that for for $i>i_0$, $K\subset\Omega_i\setminus U$. By Schauder's estimate, $w_i$ $C^{2,\alpha}$ norm is bounded above by their boundary condition and $C^0$ norm. Since $0<w_i\le u_i$, we have for any $i\ge i_0$, $||w_i||_{C^{2,\alpha}}\le C||u_0||_{C^0}$. By compact embedding of Holder space \cite[Lemma 6.33]{GT}, there exists a $C^{2,\beta}$ limit solution $w'$ where $0<\beta<\alpha$. Since the above construction is true for any compact set, we can conclude the limit function $w'$ belongs to $C^{2,\beta}_{\mathrm{loc}}(M\setminus U)$. Define $w:=-w'$, we have $-u_0\le w<0$ and $-\gamma\Delta_f w+qw=-\eta w>0$ in $M\setminus \overline{U}\supset M\setminus B(x,r) $. Then extend $w$ in $U$ such that $-2u_0\le w<0$
\end{proof}

The next lemma proves auxiliary density estimates for $\mu$-bubbles.
We refer the reader to \cite{AmbrosioFuscoPallara, MaggiBook} for the classical theory of sets of finite perimeter in the Euclidean space, which naturally extends to the case of Riemannian manifolds. 

A \textit{set of locally finite perimeter} in a complete $n$-dimensional Riemannian manifold $(M,g)$ is a measurable set $E\subset M$ such that the characteristic function $\chi_E$ has locally bounded variation.
Given such a set $E$, the total variation of the weak gradient of $\chi_E$ is called the perimeter measure, and usually denoted by $P(E,\cdot)$. By the well-known structure theorem of De Giorgi, the perimeter is concentrated on the so-called reduced boundary $\p^*E$. More precisely, we have $P(E,\cdot)=\mathcal{H}^{n-1}\llcorner\p^*E$. For ease of notation, for any set of locally finite perimeter $E\subset M$ and for any Borel set $A\subset M$, in this paper we will denote $|\partial^* E \cap A|:= P(E, A)$ and $|\partial^* E| := P(E,M)$.

\begin{lemma}\label{lem:DensityEstimatesConvergence}
Let $(M^n,g)$ be a smooth noncompact complete $n$-dimensional Riemannian manifold, with $n \ge 2$. Let $\gamma\ge 0$, $\e\in(0,\frac12)$, and let $u:M\to (0,\infty)$ be a function of class $C^{2,\alpha}$, for some $\alpha \in (0,1)$. Let $\phi:M\to \R$ be a surjective smooth proper $1$-Lipschitz function such that $\Omega_0:= \phi^{-1}\left((-\infty,0)\right)$ has smooth bounded boundary. Let $\bar h^q_\e$ be given by lemma \ref{lemma:mu_bubble_function}, and let $h^q_\e := c \,\bar h^q_\e \circ \phi$ for some $c,q>0$. Define
\begin{equation*}
    \mathcal P_\e (E) := \int_{\partial^*E} u^\gamma\, e^{-f} - \int \left( \chi_{E}- \chi_{\Omega_0} \right) h_\e u^\gamma\, e^{-\left(\frac{n+3}{n+1}\right)f},
\end{equation*}
for any \textit{admissible} set of locally finite perimeter $E\subset M$, i.e., such that $E \Delta \Omega_0 \Subset \phi^{-1}\left((-q/\e, q/\e)\right)$ up to negligible sets.

Assume that $\Omega_\e$ is a minimizer for $\mathcal P_\e$ with respect to admissible compactly supported variations. Then, up to modify $\Omega_\e$ on a negligible set, the following holds.

\begin{enumerate}
    \item The reduced boundary $\partial^* \Omega_\e$ 
    is an (open) hypersurface of class $C^{2,\alpha'}$, 
    for some $\alpha'\in(0,1)$, and $\partial^* \Omega_\e\setminus \partial^* \Omega_\e$ 
    is a closed set with Hausdorff dimension less than or equal to $n-8$.

    \item 
    Let $A\subset M$ be a precompact open set, and let $\e_0\in(0,1/2)$ such that $\overline{A}\subset \phi^{-1}\left((-1/\e_0, 1/\e_0)\right)$. Then there exist $r_0,C_d \in (0,1)$ depending on $n,\gamma, \|\log u\|_{L^\infty(A)}, g|_{\overline{A}}, \e_0, c,||f||_{0},q$ such that for any $\e<\e_0$ there holds 
    $$C_d \le \frac{|\p^*\Omega_\e \cap B(y,\rho)|}{\rho^{n-1}} \le C_d^{-1},
    $$
    for any $y \in A \cap \p^*\Omega_\e$ with $B(y,2r_0)\Subset A$, and any $\rho<r_0$.
\end{enumerate}
\end{lemma}

\begin{proof}
The first item just follow \cite[Theorem 27.5]{MaggiBook} and \cite[Theorem 28.1]{MaggiBook}. For the second item, denote $C>0$ is a constant depending on $n,N,\gamma,||f||_0,q,||\log u||_{L^\infty(A)},g|_{\overline{A}},\e_0,c$, that change from line to line. Ley $y\in A$ and $\rho>0$ such that $B(y,\rho)\subset A$. Since $h_\e\to 0$ smoothly on $\overline{A}$ as $\e\to 0$, we have $||h_\e||_{L^\infty(\overline{A})}\le C$ for any $\e<\e_0$. For any arbitrary set $F$ of locally finite perimeter such that $F\Delta \O_\e\subset B(y,\rho)$, by the minimizer, we have \begin{equation}
    \begin{split}
        \int_{\p^*\O}u^\gamma e^{-f}&\le \int_{\p^* F}u^\gamma+\int h_\e u^{\gamma}(\chi_{\O_\e}-\chi_F) e^{-\left(\frac{n+3}{n+1}\right)f}\\
        &\le \int_{\p^* F}u^\gamma+||h_\e||_{L^\infty(A)}||u||^\gamma_{L^\infty(A)} e^{100n||f||_0}|F\Delta \O_\e|\\
        &\le\int_{\p^* F}u^\gamma+C|F\Delta \O_\e|
    \end{split}
\end{equation}By the upper and lower bound of $u$ in $B(y,\rho)$, we have \begin{equation}
    \begin{split}
        \int_{\p^*\O_\e\setminus \p^*F}u^\gamma\ge \int_{(\p^*\O_\e\setminus \p^*F)\cap B(y,\rho)}u^\gamma\ge \int_{\p^*\O_\e\cap B(y,\rho)}u^\gamma-\int_{\p^*F\cap B(y,\rho)}u^\gamma
    \end{split}
\end{equation}Therefore, we have \begin{equation}
   C|\p^*\O_\e\cap B(y,\rho)| \le\int_{\p^*\O_\e\cap B(y,\rho)}u^\gamma\le \int_{\p^*F\cap B(y,\rho)}u^\gamma+C|F\Delta \O_\e|\le C(|\p^*F\cap B(y,\rho)|+|F\Delta \O_\e|)
\end{equation}Since we are working in a pre-compact set $A$, $|\n f|,f$ are bounded in $A$. Therefore, the local isoperimetric inequality with Euclidean exponents for sets contained in $A$ and the Ahlofs-type bound contained in $A$ are still hold in weighted measure. Therefore, we have\begin{equation}
    \begin{split}
    \big|\partial^*\Omega_\e \cap B(y,\rho)\big|
    &\leqslant C \left(\big|\partial^* F \cap B(y,\rho)\big| + \big|F\Delta \Omega_\e\big|^{\frac1n} \big|F\Delta \Omega_\e\big|^{\frac{n-1}{n}} \right) \\
    &\leqslant C \Big(\big|\partial^* F\cap  B(y,\rho)\big| + \big|B(y,\rho)\big|^{\frac1n}\cdot C_{\rm iso} \big|\partial^* ( F\Delta \Omega_\e)\big| \Big)
    \\
    &\leqslant C \Big( \big|\partial^* F\cap B(y,\rho)\big| + C \rho \Big[ \big|\partial^* \Omega_\e\cap B(y,\rho)\big| + \big|\partial^* F \cap B(y,\rho)\big| \Big] \Big),
\end{split}
\end{equation} 
Then, $\Omega_\e$ is a $C$-quasiminimal set in balls $B(y,2r_0) \Subset A$ in the sense of \cite[Definition 3.1]{KinnunenShanmugalingamQuasiminimalSets}. The claimed density estimates then follow from
\cite[Lemma 5.1]{KinnunenShanmugalingamQuasiminimalSets} and Item (1).
\end{proof}
\subsection{Quantitative estimate for quadratic polynomial}\label{subsec: Quantitative Estimate for Quad}
In this section, we consider the following quadratic polynomial such that $a,c>0$ and $b\in \R$ with negative discriminant.\begin{equation}\label{eq:quadratic polynomial}
    -ax^2+bx-c\tab \Delta=b^2-4ac<0
\end{equation}We are looking for the range and the relationship between $d_0,d_1>0$ such that the following inequality hold \begin{equation}\label{eq: main inequality in appendix}
    -ax^2+bx-c\le -d_0 x^2-d_1
\end{equation}
\begin{theorem}\label{thm:polynomial estimate}
    Let $p(x)$ be a quadratic polynomial with above condition in (\ref{eq:quadratic polynomial}). Let $d_0,d_1>0$ such that (\ref{eq: main inequality in appendix}) hold. Then we have the following range \begin{itemize}
        \item[(i)] $0<d_0<a$ and $0<d_1<c$
        \item[(ii)]Denote $\Delta=-4\Psi$ such that $\Psi>0$. Assuming item(i) and $cd_0<\Psi/2$, such that the following inequality hold for any $d_1$ \begin{equation}
            0<d_1\le \dfrac{cd_0-\Psi}{d_0-a}
        \end{equation}then $(\ref{eq: main inequality in appendix})$ holds.
        \item[(iii)]Consider $0<\bar{d}<\min\left\{\frac{a}{2},\frac{\Psi}{2c},\frac{a\Psi}{b^2}\right\}$, such that for any $d_0<\bar{d}$,if item(i) and the following inequality hold \begin{equation}
            0<d_1\le \dfrac{\Psi}{a}-\dfrac{b^2}{a^2}d_0
        \end{equation}then (\ref{eq: main inequality in appendix}) holds.
    \end{itemize}
\end{theorem}
\begin{proof}
    Consider the following quadratic polynomial $p(x)$, \begin{equation}
        p(x):=(d_0-a)x^2+bx+(d_1-c)
    \end{equation}If we want $p(x)\le 0$, we need it is concave downward with negative $y$-intercept because $b$ might be non-zero. This condition gives item $(i)$

    To guarantee the non-negativity, we need its discrimiant is non-positive so we have\begin{equation}
        \begin{split}
            &\quad b^2-4(d_0-a)(d_1-c)\\
            &=b^2-4ac-4(d_0d_1-cd_0-ad_1)\\
            &=-4\Psi-4(d_0d_1-cd_0-ad_1)\le 0
        \end{split}
    \end{equation}By grouping the term and note that $d_0-a<0$, we have \begin{equation}
        d_1\le \dfrac{cd_0-\Psi}{d_0-a}
    \end{equation}The existence of a positive constant $d_1$ is guaranteed by the assumption $c d_0<\frac{\Psi}{2}$. This gives item$(ii)$

    For item$(iii)$, we consider the function $f:[0,\bar{d}]\to \R$ such that \begin{equation}
        f(x)=\dfrac{cx-\Psi}{x-a}
    \end{equation}The first derivative is \begin{equation}
        f'(x)=\dfrac{c(x-a)-(cx-\Psi)}{(x-a)^2}=\dfrac{\Psi-ac}{(x-a)^2}=-\dfrac{b^2}{4(x-a)^2}\le 0
    \end{equation}Therefore, $f$ is monotone decreasing for $x\in[0,\bar{d}]$ and we have \begin{equation}
        \max_{x\in [0,\bar{d}]}|f'(x)|= \max_{x\in [0,\bar{d}]}\dfrac{b^2}{4(x-a)^2}\le \dfrac{b^2}{a^2}
    \end{equation} The last inequality take $x=\frac{a}{2}$. By Taylor's theorem, we have \begin{equation}\label{eq: App Lem (iii)}
        f(d_0)\ge f(0)-d_0 \max_{x\in [0,\bar{d}]}|f'(x)|=\dfrac{\Psi}{a}-\dfrac{b^2}{a^2}d_0
    \end{equation}By item $(ii)$ and \ref{eq: App Lem (iii)}, we have \begin{equation}
        0<d_1\le \dfrac{\Psi}{a}-\dfrac{b^2}{a^2}d_0
    \end{equation}The positivity is guranteed by $d_0<\frac{a\Psi}{b^2}$. This gives item$(iii)$

\end{proof}

\section{Spectral Splitting theorem in BE condition}\label{sec:main thm}
In this section, we are going to prove the Theorem \ref{thm:main_thm}.  If $\gamma\le 0$ the statement reduces to the classical $N$-Bakry \'Emery splitting theorem. Hence, we assume $\gamma>0$.
\subsection{\texorpdfstring{$\mu$}{} bubble construction}
    By Theorem \ref{thm:equiv+positive}, and since $\Ric^N_f\in Lip_{loc}(M)$, the main condition $\ref{eq:main condition}$, implies the existence of a function $u_0\in C^{2,\a}(M)$ with \begin{equation}
    u_0>0,\tab -\gamma\Delta_f u_0+\Ric^N_f\cdot u_0=0
\end{equation}Fix an arbirary $x\in M$. We claim that $|\n u_0|(x)=0$. Once this is proved, since $x$ is arbitrary, it follows that $u_0$ is constant on the entire $M$, and thus $\Ric^N_f\ge 0$, and then the theorem is reduced to the splitting theorem in $\Ric^N_f\ge 0$ by Theorem \ref{thm:splitting theorem BE}. 

 Since $M$ has more than one end, we can find a smooth surjective proper function $\phi:M\to\R$ such that
    \begin{equation}\label{eq:phi}
        |\n\phi|\leq1,\qquad |\phi(x)|<1/2,\qquad \Omega_0:=\phi^{-1}\big((-\infty,0)\big)\text{ has smooth boundary.}
    \end{equation}
    The construction is as follows. First, we find a smooth surjective proper $1$-Lipschitz function $\tilde\phi:M\to\R$. Indeed, given a separating compact hypersurface bounding a distinguished end, $\tilde\phi$ can be obtained by smoothing the signed distance function from such end as in \cite[Lemma 2.1]{Zhu23JDG}. Let $L$ be a regular value of $\tilde\phi$ and $q:=e^{\frac{2}{n-1}||f||_{0}}$ such that $|\tilde\phi(x) - L| < q/2$. Hence, $\phi:= \tilde\phi - L$ has the desired properties.

    Let $0<\delta,r\ll1$ be such that
    \begin{equation}\label{eq:radii_constraints}
        B(x,1/\delta)\Supset\phi^{-1}\big([-2q,2q]\big),\qquad
        B(x,2r)\Subset\phi^{-1}\big([-q,q]\big).
    \end{equation}
    Let us apply Lemma \ref{lemma:perturb_new} to the ball $B(x,r)$, with data $u_0$ and $f=\Ric^N_f$. Denote by $w_r$ the resulting function.
    For a parameter $a>0$, set $u_{r,a}:=u_0+aw_r$. There exists a small $a_0=a_0(\delta,w_r,u_0,N,||f||_{0})$ so that for all $a<a_0$, we have
    \begin{equation}\label{eq:choice_a_1}
        u_0/2\leq u_{r,a}\leq u_0,\qquad ||u_{r,a}-u_0||_{C^2(B(x,1/\delta))}<\delta,
    \end{equation}
    and
    \begin{equation}\label{eq:choice_a_2}
        \left\{\begin{aligned}
            & -\gamma\Delta_f u_{r,a}+\Ric^N_f\cdot u_{r,a}>0\qquad\text{in $M\setminus B(x,r)$,} \\
            & -\gamma\Delta_f u_{r,a}+\Ric^N_f\cdot u_{r,a}\geq-\delta \cdot u_{r,a}
            \qquad\text{in $B(x,r)$.} 
        \end{aligned}\right.
    \end{equation}
    In particular, in light of \eqref{eq:radii_constraints} and \eqref{eq:choice_a_2}, we have
    \begin{equation}\label{eq:choice_a_3}
        -\gamma\Delta_f u_{r,a}+\Ric^N_f\cdot u_{r,a}\geq\mu\cdot u_{r,a}\qquad\text{in $\phi^{-1}\big([-2q,2q]\big)\setminus B(x,2r)$,}
    \end{equation}for some constant $\mu=\mu(a,w_r,u_0,N,||f||_{0})>0$.

    For a constant $c_0=c_0(n,\gamma,N)>0$, that we defined in $(\ref{eq:estimate lemma})$, and for any $\e<1/10$, let $h_\e$ be defined as     \begin{equation}\label{eq:mu_e function}
        h^q_\e(x):=c_0^{-1}\bar h^q_\e(\phi(x)),\qquad\forall x\in\phi^{-1}\big((-q/\e,q/\e)\big),
    \end{equation}
 where $\bar h^q_\e$ is as in Lemma \ref{lemma:mu_bubble_function}. We will simply denote $h_\e:=h^{q}_\e$ In particular, on $\phi^{-1}\big((-q/\e,-q]\cup[q,q/\e)\big)$ we have 
    \begin{equation}\label{eq:dh+h2}
        \begin{split}
                c_0h_\e^2 - q|\n h_\e| &= 
                c_0^{-1}\Big[\bar h_\e(\phi)^2+q(\bar h^q_\e)'(\phi)||\nabla\phi|\Big]
                \geq c_0^{-1}\e^2,
        \end{split}
    \end{equation}
    where we used that $\phi$ is $1$-Lipschitz and Lemma \ref{lemma:mu_bubble_function}(1)(3). Then by Lemma \ref{lemma:mu_bubble_function}(2)(3), inside $\phi^{-1}\big([-q,q]\big)$ we have
    \begin{equation}\label{eq:dh+h2_2}
        c_0h_\e^2-|\n h_\e|\geq-Cc_0^{-1}\e,
    \end{equation}
    where $C$ is the constant in \ref{lemma:mu_bubble_function}(2). For $\delta,r,a$ small enough as indicated above, we perform the following $\mu$-bubble argument. For notational simplicity, we denote $u=u_{r,a}$. For any set of finite perimeter $\Omega$ with $\Omega\Delta \Omega_0 \subset \phi^{-1}\big((-q/\e,q/\e)\big)$, define the energy
\begin{equation}
    \mathcal{P}_\e(\Omega)=\int_{\p^*\Omega}u^{\gamma}e^{-f}-\int(\chi_{\Omega}-\chi_{\Omega_0})hu^\gamma e^{-(k+1)f}
\end{equation}

As argued in \cite[Proposition 12]{ChodoshliSoapBubbles24}, recalling that $h_\varepsilon$ diverges as $\phi\to \pm q/\e$ by Lemma \ref{lemma:mu_bubble_function}(3), there exists a minimizer $\Omega_\e$ for $\mathcal{P}_\e$ with $\Omega_\e\Delta \Omega_0 \Subset \phi^{-1}\big((-q/\e,q/\e)\big)$. 

 Since $u\in C^{2,\alpha}$, the regularity result in Lemma \ref{lem:DensityEstimatesConvergence}(1) applies to $\Omega_\e$.
    Let $\nu$ denote the outer unit normal at $\partial^*\Omega_\e$, and let $\varphi\in C^\infty(M)$. For an arbitrary smooth variation $\{F_t\}_{t\in(-t_0,t_0)}$, with $F_0=\Omega_\e$, and with variation field equal to $\varphi\nu$ at $t=0$.\footnote{When $n\ge 8$, the boundary $\p^*\Omega_\e$ may contain a nonempty codimension 8 singular set $\p^*\Omega_\e \setminus \p^*\Omega_\e$, see lemma \ref{lem:DensityEstimatesConvergence}(1). However, arguing as in \cite[Appendix A]{AX24}, multiplying $\varphi$ by a cut off function vanishing on $\p^*\Omega_\e \setminus \p^*\Omega_\e$, one can calculate the first and second variation of $\mathcal P_\e$, and carry out computations leading to (\ref{eq: 2nd variation final form}). For the sake of completeness, We present}
\subsection{Geometric Estimate for \texorpdfstring{$\mu$}{}-bubbles}
\subsubsection{First variation}
 We compute the first variation\begin{equation}
    0 = \frac{d}{dt^2}\mathcal{P}_\e(F_t)\Big|_{t=0}
= \int_{\p^*\Omega}(H+\gamma u^{-1}u_\nu-f_\nu-he^{-kf})u^\gamma e^{-f}\varphi.
\end{equation}
Since $\vp$ is arbitrary, then we obtain
\begin{equation}\label{H}
H = f_\nu+he^{-kf}-\gamma u^{-1}u_\nu.
\end{equation}
\subsubsection{Second variation}We next compute the second variation:
\begin{equation}
    \begin{split}
0 \leq {} & \frac{d^2}{dt^2}\mathcal{P}_\e(F_t)\Big|_{t=0} \\
= {} & \int_{\p^*\Omega}\Big(-\Delta_{\p^*\Omega}\varphi-|\mathrm{II}|^2\varphi-\Ric(\nu,\nu)\varphi-\gamma u^{-2}u_\nu^2\varphi \\
& \quad \quad \ \, +\gamma u^{-1}\varphi(\Delta u-\Delta_{\p^*\Omega}u-Hu_\nu)-\gamma u^{-1}\langle\nabla_{\p^*\Omega}u,\nabla_{\p^*\Omega}\varphi\rangle \\[1mm]
& \quad \quad \ \, -\mathrm{Hess} f(\nu,\nu)\vp + \langle\nabla_{\p^*\Omega} f,\nabla_{\p^*\Omega}\varphi\rangle-h_{\nu}e^{-kf}\varphi\\
& \quad \quad \ \,+k he^{-kf}f_{\nu}\varphi\Big) u^\gamma e^{-f}\varphi.
\end{split}
\end{equation}
By the definition of $\Delta_{f}$ and $\Ric_{f}^{N}$,
\begin{equation}\label{Ric Delta u Hess f}
\begin{split}
& -\Ric(\nu,\nu)\varphi+\gamma u^{-1}\varphi\Delta u-\mathrm{Hess} f(\nu,\nu)\vp \\[1mm]
= {} & -\big(\Ric_{f}^{N}(\nu,\nu)+\frac{1}{N}f_{\nu}^{2}\big)\varphi+\gamma u^{-1}\varphi\big(\Delta_{f}u+\langle\nabla u,\nabla f\rangle\big) \\
= {} & \big(\gamma\Delta_{f}u-u\Ric_{f}^{N}(\nu,\nu)\big)u^{-1}\varphi-\frac{1}{N}f_{\nu}^{2}\vp+\gamma u^{-1}\varphi\langle\nabla u,\nabla f\rangle.
\end{split}
\end{equation}
It then follows that
\begin{equation}
    \begin{split}
0 \leq \int_{\p^*\Omega}
\Big(& -\Delta_{\p^*\Omega}\varphi-|\mathrm{II}|^2\varphi+\big(\gamma\Delta_{f}u-u\Ric_{f}^{N}(\nu,\nu)\big)u^{-1}\varphi-\frac{1}{N}f_{\nu}^{2}\vp \\
& +\gamma u^{-1}\varphi\langle\nabla u,\nabla f\rangle-\gamma u^{-2}u_\nu^2\varphi+\gamma u^{-1}\varphi(-\Delta_{\p^*\Omega}u-Hu_\nu) \\[1mm]
& -\gamma u^{-1}\langle\nabla_{\p^*\Omega}u,\nabla_{\p^*\Omega}\varphi\rangle
+\langle\nabla_{\p^*\Omega} f,\nabla_{\p^*\Omega}\varphi\rangle-h_{\nu}e^{-kf}\varphi\\
& +khe^{-kf}f_{\nu}\varphi\Big) u^\gamma e^{-f}\varphi.
\end{split}
\end{equation}Since $\p^*\Omega_\e$ is compact, by taking $\varphi=u^{-\frac{\gamma}{2}}$. We consider the following lemma, which is inspired by \cite{Chu_BE}
\begin{lemma}
    When $\varphi=u^{-\frac{\gamma}{2}}$, we have \begin{equation}
(\gamma u^{-1}\varphi\<\n u,\n f\>+\<\n_{\p^*\Omega}f,\n_{\p^*\Omega}\varphi\>)u^{\gamma}\varphi=\gamma u^{-1}u_\nu f_\nu+\dfrac{\gamma}{2}u^{-1}\<\n_{\p^*\Omega}f,\n_{\p^*\Omega} u\>
    \end{equation}
\end{lemma}
\begin{proof}
    The result is following by computation. It is not hard to notices that $$u^\gamma\cdot u^{-\frac{\gamma}{2}}=u^{\frac{\gamma}{2}}=\varphi^{-1}$$Therefore, we have \begin{equation}
        \begin{split}
            &(\gamma u^{-1}\varphi\<\n u,\n f\>+\<\n_{\p^*\Omega}f,\n_{\p^*\Omega}\varphi\>)u^{\gamma}\varphi\\
            ={}&(\gamma u^{-1}\varphi\<\n u,\n f\>+\<\n_{\p^*\Omega}f,\n_{\p^*\Omega}\varphi\>)\varphi^{-1}\\
            ={}&\gamma u^{-1}\<\n u,\n f\>+u^{\frac{\gamma}{2}}\<\np f,\np u^{-\frac{\gamma}{2}}\>\\
            ={}&\gamma u^{-1}\<\n u,\n f\>-\frac{\gamma}{2}u^{-1}\<\n_{\p^*\Omega}f,\np u\>\\
            = {}&\gamma u^{-1}u_\nu f_\nu+\dfrac{\gamma}{2}u^{-1}\<\np f,\np u\>
        \end{split}
    \end{equation}
\end{proof}
As a result, we obtain,\begin{equation}\label{eq: 2nd variation final form}
  \begin{split}
        0\le& \dfrac{d^2}{dt^2}\mathcal{P}_\e(F_t)\Big|_{t=0}\\
        \le&\int_{\p^*\Omega}\Big[-u^{\frac{\gamma}{2}}\Delta_{\p^*\Omega}u^{-\frac{\gamma}{2}}-\gamma u^{-1}\Delta_{\p \O}u-\gamma u^{\frac{\gamma}{2}-1}\<\n_{\p^*\Omega }u,\n_{\p \O}u^{-\frac{\gamma}{2}}\>+\frac{\gamma}{2}\<\n_{\p^*\Omega}f,\n_{\p^*\Omega}u\>u^{-1}\Big]e^{-f}\\
        &\quad+\Big[-|\mathrm{II}|^2-\dfrac{1}{N}f^2_\nu-\gamma u^{-2}u^2_\nu-\gamma u^{-1}Hu_\nu-h_\nu e^{-kf}+\gamma u^{-1}u_\nu f_\nu+kh e^{-k f} f_\nu\Big]e^{-f}\\
        &\quad +\Big[\gamma u^{-1} \Delta_f u-\Ric^N_f(\nu,\nu)\Big]e^{-f}\\
        :=&\int (P+Q+R) e^{-f}\dvolg
  \end{split}
\end{equation}Denote $P$ be the tangential part, $Q$ be the normal part and $R$ be the spectral part.\\
\subsubsection{Tangential Component Estimate}
We consider the tangential part, by straight forward computation to each term.
\begin{equation}\label{eq: P1}
\begin{split}
        \int_{\p^*\Omega}(-u^{\frac{\gamma}{2}}\Delta_{\p^*\Omega}u^{-\frac{\gamma}{2}})e^{-f}&= \int_{\p^*\Omega}\<\n_{\p^*\Omega}u^{-\frac{\gamma}{2}},\n_{\p^*\Omega}(e^{-f}u^\frac{\gamma}{2})\>\\
        &=\int_{\p^*\O}\<-\frac{\gamma}{2}u^{-\frac{\gamma}{2}-1}\n_{\p^*\Omega}u,-e^{-f}u^\frac{\gamma}{2}\n_{\p^*\Omega} f+e^{-f}u^{\frac{\gamma}{2}-1}\frac{\gamma}{2}\n_{\p^*\Omega} u\>\\
        &=\int_{\p^*\O}\Big[\dfrac{\gamma}{2}u^{-1}\<\n_{\p^*\Omega}u,\n_{\p^*\Omega} f\>-\left(\dfrac{\gamma}{2}\right)^2u^{-2}|\n_{\p^*\Omega} u|^2\Big] e^{-f}
\end{split}
\end{equation}
By similar computation, we have
\begin{equation}\label{eq: P2}
\int_{\p^*\O}(-\gamma u^{-1}e^{-f})\Delta_{\p^*\Omega}u=\int_{\p^*\O}\Big[-\gamma u^{-2}|\np u|^2-\gamma u^{-1}\<\np f,\np u\>\Big] e^{-f}
\end{equation}
By Chain rule, \begin{equation}\label{eq:P3}
    \int_{\p^*\O}-\gamma u^{\frac{\gamma}{2}-1}\<\n_{\p^*\Omega }u,\np u^{-\frac{\gamma}{2}}\>e^{-f}=\int_{\p^*\O}\left(\dfrac{\gamma^2}{2}\right)u^{-2}|\np u|^2 e^{-f}
\end{equation}Combining the equation $(\ref{eq: P1}),(\ref{eq: P2}),(\ref{eq:P3})$, we obtain for $0< \gamma\le 4$ \begin{equation}\label{eq: tangential estimate}
    \begin{split}
        &\int P\, e^{-f}\dvolg\\
        = {} &\int_{\p^*\Omega}\Big[\dfrac{\gamma}{2}u^{-1}\<\n_{\p^*\Omega}u,\n_{\p^*\Omega} f\>-\left(\dfrac{\gamma}{2}\right)^2u^{-2}|\n_{\p^*\Omega} u|^2-\gamma u^{-2}|\np u|^2-\gamma u^{-1}\<\np f,\np u\>\\
         {}  & \quad\quad\quad\hspace{-1mm} +\left(\dfrac{\gamma^2}{2}\right)u^{-2}|\np u|^2+\dfrac{\gamma}{2}u^{-1}\<\np f,\np u\>\Big]e^{-f}\dvolg\\
         = {} &\int_{\p^*\O}\gamma\left(\dfrac{\gamma}{4}-1\right)u^{-2}|\np u|^2 e^{-f}\dvolg\\
         \le {}& 0
    \end{split}
\end{equation}
\subsubsection{Normal Component Estimate}

For the normal part, by $|\mathrm{II}|^2\geq \frac{H^2}{n-1}$, we have \begin{equation}
    \begin{split}
\int_{\p^*\Omega} Q\, e^{-f}\dvolg \leq \int_{\p^*\Omega}
\Big(& -\frac{H^2}{n-1}-\frac{1}{N}f_{\nu}^{2}+\gamma u^{-1}u_{\nu}f_{\nu}\\
& -\gamma u^{-2}u_\nu^2-\gamma u^{-1}Hu_\nu -h_{\nu}e^{-kf}\\
&+khe^{-kf}f_{\nu}\Big)e^{-f} \\
\leq \int_{\p^*\Omega}
\Big(& -\frac{H^2}{n-1}-\frac{1}{N}f_{\nu}^{2}+\gamma u^{-1}u_{\nu}f_{\nu}\\
& -\gamma u^{-2}u_\nu^2-\gamma u^{-1}Hu_\nu +|\nabla h|e^{-kf}\\
&+khe^{-kf}f_{\nu}\Big)e^{-f}.
\end{split}
\end{equation}
Set $X=he^{-kf}$ and $Y=u^{-1}u_{\nu}$. By the construction of $\mu$-Bubble, we have $H=f_{\nu}+X-\gamma Y$.
Direct computations show,
\begin{equation}
    \begin{split}
0 \leq \int_{\p^*\Omega}\bigg[& -\frac{(f_{\nu}+X-\gamma Y)^2}{n-1}-\frac{1}{N}f_{\nu}^2 \\
& +\gamma Yf_\nu-\gamma Y^2  -\gamma Y(f_\nu+X-\gamma Y)\\
&+|\nabla h|e^{-kf}+kXf_\nu  \bigg] e^{-f} \\
= \int_{\p^*\Omega}
\bigg[& -\Big(\frac{1}{n-1}+\frac{1}{N}\Big)f_{\nu}^2+\Big(k-\frac{2}{n-1}\Big)Xf_{\nu}+\frac{2\gamma}{n-1}Yf_{\nu} \\
& -\frac{X^{2}}{n-1}+\Big(\frac{2\gamma}{n-1}-\gamma\Big)XY+\Big(\frac{n-2}{n-1}\gamma^2-\gamma\Big)Y^2 \\
& +|\nabla h|e^{-kf}\bigg]e^{-f}.
\end{split}
\end{equation}Consider the following $(X,Y)$-polynomial \begin{equation}
    -\frac{X^{2}}{n-1}+\Big(\frac{2\gamma}{n-1}-\gamma\Big)XY+\Big(\frac{n-2}{n-1}\gamma^2-\gamma\Big)Y^2 
\end{equation}Computing the discriminant, we get \begin{equation}\label{eq:first discriminant}
    \Delta=\Big(\frac{2\gamma}{n-1}-\gamma\Big)^2-4\left(-\frac{X^{2}}{n-1}\right)\Big(\frac{n-2}{n-1}\gamma^2-\gamma\Big)=\gamma\left(\gamma-\dfrac{4}{n-1}\right)<0
\end{equation}when $0<\gamma<\frac{4}{n-1}$.\\Therefore, there exists $c_0,c_1>0$ such that \begin{equation}\label{eq:estimate lemma}
     -\frac{X^{2}}{n-1}+\Big(\frac{2\gamma}{n-1}-\gamma\Big)XY+\Big(\frac{n-2}{n-1}\gamma^2-\gamma\Big)Y^2 \le -c_0  X^2-c_1 Y^2
\end{equation}By the quantitative estimate in theorem \ref{thm:polynomial estimate}, there exists $0<\bar{c}=\bar{c}(n,\gamma)\le \frac{1}{2(n-1)}$ such that for any $c_0<\frac{1}{n-1}$, $c_1<\frac{n-2}{n-1}\gamma^2-\gamma$ and $0<c_0<\bar{c}$, we can choose \begin{equation}\label{eq:apply quan estimate}
    c_1=\Psi(n-1)-(n-1)^2\left(\frac{2\gamma}{n-1}-\gamma\right)^2c_0=\left(\dfrac{n-1}{4}\right)\gamma\left(\frac{4}{n-1}-\gamma\right)-(n-3)^2\gamma^2c_0
\end{equation} such that (\ref{eq:estimate lemma}) holds.\\If we further choose $k=\frac{2}{n-1}$, then the normal part $Q$ will have the following expression. 
\begin{equation}
    \begin{split}
0 \leq \int_{\p^*\Omega}
\bigg[& -\Big(\frac{1}{n-1}+\frac{1}{N}\Big)f_{\nu}^2+\frac{2\gamma}{n-1}Yf_{\nu} \\
& -c_0 X^2-\left[\left(\dfrac{n-1}{4}\right)\gamma\left(\frac{4}{n-1}-\gamma\right)-(n-3)^2\gamma^2c_0\right]Y^2 \\
& +|\nabla h|e^{-kf}\bigg]e^{-f}.
\end{split}
\end{equation}Consider the $(f_\nu,Y)$-polynomial, \begin{equation}
\begin{split}
        p(f_\nu,Y):=&-\Big(\frac{1}{n-1}+\frac{1}{N}\Big)f_{\nu}^2+\frac{2\gamma}{n-1}Yf_{\nu}\\
        &-\left[\left(\dfrac{n-1}{4}\right)\gamma\left(\frac{4}{n-1}-\gamma\right)-(n-3)^2c_0\gamma^2\right]Y^2
\end{split}
\end{equation}By completing the square, we have
\begin{equation}
    \begin{split}
         &\quad-\Big(\frac{1}{n-1}+\frac{1}{N}\Big)f_{\nu}^2+\frac{2\gamma}{n-1}Yf_{\nu}-\left[\left(\dfrac{n-1}{4}\right)\gamma\left(\frac{4}{n-1}-\gamma\right)-(n-1)^2\left(\frac{2\gamma}{n-1}-\gamma\right)c_0\right]Y^2\\
         &=-\left(\dfrac{1}{n-1}+\frac{1}{N}\right)\left(f_\nu-\left(\dfrac{1}{n-1}+\dfrac{1}{N}\right)^{-1}\dfrac{\gamma}{n-1}Y\right)^2\\
         &\,\,\quad+\gamma\left[\left(\dfrac{1}{(n-1)+(n-1)^2/N}+\frac{n-1}{4}\right)\gamma-1+(n-3)^2\gamma c_0\right]Y^2
    \end{split}
\end{equation}Denote $$c_{N,n}=\dfrac{1}{(n-1)+(n-1)^2/N}+\frac{n-1}{4}$$
To make the last term negative, we assume the range for $\gamma$ is \begin{equation}\label{eq:final range of gamma}
    0<\gamma<\frac{1}{c_{N,n}+(n-3)^2c_0}
\end{equation}
When $n=3$, we get the upperbound of $\gamma$ in $(\ref{eq:final range of gamma})$ which we want in Theorem $\ref{thm:main_thm}$.

We consider $n\neq 3$, but the choice of $c_0$ is depends on $\gamma$. To find a range of $c_0$ that is independent of $\gamma$. We first observe that $(n-3)^2c_0\ge 0$, then $\gamma$ has an upper bound $$    0<\gamma<\frac{1}{c_{N,n}+(n-3)^2c_0}<c_{N,n}^{-1}$$
By the quantitative estimate in Theorem $\ref{thm:polynomial estimate}$, we have the explicit expression of $\bar{c}$ in the following by matching the coefficient\begin{equation}\label{eq:coefficient of bar c}
    \begin{split}
        \bar{c}&=\min\left\{\dfrac{a}{2},\dfrac{\Psi}{2c},\dfrac{a\Psi}{b^2}\right\}\\
        &=\min\left\{\dfrac{1}{2(n-1)},\dfrac{\left(\frac{4}{n-1}-\gamma\right)}{4\left(1-\frac{n-2}{n-1}\gamma\right)},\dfrac{\left(\frac{4}{n-1}-\gamma\right)(n-1)}{4\gamma(n-3)^2}\right\}\\
        &\ge\min\left\{\dfrac{1}{2(n-1)},\dfrac{\left(\frac{4}{n-1}-c_{N,n}^{-1}\right)}{4},\dfrac{\left(\frac{4}{n-1}-c_{N,n}^{-1}\right)(n-1)}{4c_{N,n}^{-1}(n-3)^2}\right\}\\
        &=:\tilde{c}(N,n)>0
    \end{split}
\end{equation}
Therefore, we have the following range
\begin{equation}\label{eq:final range of gamma for n neq 3}
    0<\gamma<\frac{1}{c_{N,n}+(n-3)^2c_0}
\end{equation}the upperbound depends on the choice of $c_0$ quantity and $c_0<\tilde{c}_{N,n}$
Therefore, there exists a $d_0=d_0(n,N,c_0,\gamma)>0$ such that \begin{equation}
    p(f_\nu,Y)=-\left(\dfrac{1}{n-1}+\frac{1}{N}\right)\left(f_\nu-\left(\dfrac{1}{n-1}+\dfrac{1}{N}\right)^{-1}\dfrac{\gamma}{n-1}Y\right)^2-d_0 Y^2
\end{equation}Fix the $c_0$ is sufficiently small, we have the following estimate 
\begin{equation}\label{eq: final estimate for normal part}
    \begin{split}
0 \leq \int_{\p^*\Omega}
\bigg[& -\left(\dfrac{1}{n-1}+\frac{1}{N}\right)\left(f_\nu-\left(\dfrac{1}{n-1}+\dfrac{1}{N}\right)^{-1}\dfrac{\gamma}{n-1}Y\right)^2 \\
& -c_0 X^2-d_0Y^2 +|\nabla h|e^{-kf}\bigg]e^{-f}.\\
\le \int_{\p \O}&\Big[-c_0 X^2-d_0Y^2 +|\nabla h|e^{-kf}\bigg]e^{-f}
\end{split}
\end{equation}

\subsubsection{Spectral Component Estimate}
On the other hand, for the spectral term $R$, by (\ref{eq:choice_a_2}), $(\ref{eq:choice_a_3})$ and $||f||_{0}<\infty$. We have\begin{equation}\label{eq: R integral}
   \begin{split}
        \int_{\p^*\O} R\,e^{-f}\dvolg&\le \int_{\p^*\O}(\delta\chi_{B(x,r)}-\mu\chi_{\phi^{-1}([-2q,2q])\setminus B(x,2r)} )\, e^{-f}\dvolg\\
        &\le \delta e^{||f||_{0}}|\p^*\O_\e\cap B(x,r)|-\mu|\p^*\O_\e\cap(\phi^{-1}([-2q,2q]\setminus B(x,2r)))|e^{-||f||_{0}}
   \end{split}
\end{equation}

\subsection{Final Proof}

\begin{proof}
    Back to the inequality (\ref{eq: 2nd variation final form}), combining (\ref{eq: tangential estimate}), (\ref{eq: final estimate for normal part}), (\ref{eq: R integral}). We obtain
\begin{equation}\label{eq:final mu bubble estimate}
    \begin{split}
        0&\le\int_{\p^*\O}(P+Q+R) e^{-f}\dvolg\\
        &\le-c_0\int_{\p^*\O}u^{-2}|\np u|^2 e^{-f}\dvolg-d_0\int_{\p^*\O}u^{-2}u^{2}_\nu e^{-f}\dvolg\\
        &\,\,\,\,\,\,\,+\int_{\p^*\O} \Big[-c_0h^2_\e+|\np h_\e|e^{\frac{2}{n-1}f}\Big]e^{-\frac{n+3}{n-1}f}\dvolg\\
        &\quad\,+ \delta e^{||f||_{0}}|\p^*\O_\e\cap B(x,r)|
      -\mu|\p^*\O_\e\cap(\phi^{-1}([-2q,2q]\setminus B(x,2r)))|e^{-||f||_{0}}\\
      &\le-c_0\int_{\p^*\O}u^{-2}|\np u|^2 e^{-f}\dvolg-d_0\int_{\p^*\O}u^{-2}u^{2}_\nu e^{-f}\dvolg\\
      &\quad - c_0^{-1}\e^2 e^{\frac{n+3}{n-1}||f||_{0}}\big|\p^*\Omega_\e\setminus\phi^{-1}([-q,q])\big|
        + Ce^{\frac{n+3}{n-1}||f||_{0}}c_0^{-1}\e\big|\p^*\Omega_\e\cap\phi^{-1}([-q,q])\big|\\
        &\quad+ \delta e^{||f||_{0}}|\p^*\O_\e\cap B(x,r)|-\mu|\p^*\O_\e\cap(\phi^{-1}([-2q,2q]\setminus B(x,2r)))|e^{-||f||_{0}}\\
    \end{split}
\end{equation}where $C=C(q)$ is the constant defined in Lemma \ref{lemma:mu_bubble_function}(2)\\
Note that $\p^*\O_\e\cap \phi^{-1}([-q,q])\neq \emptyset$. Otherwise, $(\ref{eq:final mu bubble estimate})$ and $B(x,r)\subset \phi^{-1}([-q,q])$ together implies \begin{equation}
    0\le- c_0^{-1}\e^2 e^{\frac{n+3}{n-1}||f||_{0}}\big|\p^*\Omega_\e\setminus\phi^{-1}([-q,q])\big|<0
\end{equation}Next, we claim that $\p^*\Omega_\e\cap B(x,2r)\ne\emptyset$ whenever $\e<C^{-1}c_0\mu e^{-\left(\frac{2n+2}{n-1}\right)||f||_0}$. Indeed, since $\p^*\Omega_\e\cap\phi^{-1}\big([-1,1]\big)\ne\emptyset$, we have $|\p^*\Omega_\e\cap\phi^{-1}\big([-2q,2q]\big)|>0$. Hence if $\p^*\Omega_\e\cap B(x,2r)=\emptyset$, (\ref{eq: 2nd variation final form}) would imply\begin{equation}
\begin{split}
        0&\le Ce^{\frac{n+3}{n-1}||f||_{0}}c_0^{-1}\e\big|\p^*\Omega_\e\cap\phi^{-1}([-q,q])\big|-\mu|\p^*\O_\e\cap(\phi^{-1}([-2q,2q]\setminus B(x,2r)))|e^{-||f||_{0}}\\
        &<0
\end{split}
\end{equation}

Let now $m_0\in \mathbb N$. For every $m\geqslant m_0$, define $\delta_m:=\frac{1}{m}$, and $r_m:=\frac{1}{m}$. If $m_0$ is chosen large enough, according to what said above, for every $\delta_m,r_m$ we can find $a_m,\mu_m>0$ small enough such that \eqref{eq:radii_constraints}, \eqref{eq:choice_a_1}, \eqref{eq:choice_a_2}, and \eqref{eq:choice_a_3} are met. Denote by $u_m=u_0+a_mw_{r_m}$ the perturbed functions. Choosing $\varepsilon_m<\min\{\frac{1}{m},C^{-1}c_0\mu e^{-\left(\frac{2n+2}{n-1}\right)||f||_0}\}$, from what we said above we know that there exists $y_m\in \partial^*\Omega_{\varepsilon_m}\cap B(x,2r_m)$.

We will apply the density lemma \ref{lem:DensityEstimatesConvergence} in this argument. Notice that $y_m\to x$ such that $y_m\in \overline{B(x,r_{m_0})}$, and by $(\ref{eq:choice_a_1})$ we know $\log u_m$ is bounded uniformly in $\overline{B(x_0,r_{m_0})}$ with respect to $m$. Take $A=B(x,r_{m_0})$ such that it is a pre-compact open set and by (\ref{eq:radii_constraints}), we have $\overline{B(x,r_{m_0})}\subset \phi^{-1}((-q/\e_0,q/\e_0))$.

Choosing $m_1>4m_0$ and for any $m>m_1$. Applying the density lemma $\ref{lem:DensityEstimatesConvergence}$ in our case, there exists uniform constant $r_0,C_d\in (0,1)$ depending on $n,\gamma,||\log u||_{L^\infty(\overline{B(x,r_{m_0})})},g|_{\overline{B(x,r_{m_0})}},q,N$ such that for any $\e<\e_0$ there holds
\begin{equation}
    C_d\le\dfrac{|\p^*\Omega_{\e_m}\cap B(y_m,\rho)|}{\rho^{n-1}}\le C^{-1}_d 
\end{equation}for any $y_m\in \overline{B(x,r_{m_0})}\cap\p^*\Omega_\e$ as $B(y_m,2r_m)\subset B(x,r_{m_0})$ with $m>m_1$, and any $\rho<r_{m_0}$

Assume, $|\n u_0|(x)>0$, by $u_0\in C^{2,\a}(M)$, there exists a $\tilde{\delta}=\tilde \delta_x>0$ such that $|\n u_0|(z)>0$ for any $z\in B(x,\tilde{\delta})$. There exists a $m_1>3m_0$ such that for $m>m_1$ , $y_m\in B(x,\tilde\delta)$ and we can find a constant $\rho$ such that $0<\rho<\tilde\delta$ with $B(y_m,\rho/2)\subset B(x,\rho)\subset B(x,\tilde\delta)$. For a sufficiently large $m$, by the density estimate and $(\ref{eq:choice_a_1})$,\begin{equation}
\begin{split}
        \int_{\p^*\O_{\e_m}\cap B(x,\rho)}u^{-2}_m|\n u_m|^2 e^{-f}\dvolg&\ge\int_{\p^*\O_{\e_m}\cap B(y_m,\rho/2)}u^{-2}_m|\n u_m|^2 e^{-f}\dvolg\\
        &\overset{(\ref{eq:choice_a_1})}{\ge}\int_{\p^*\O_{\e_m}\cap B(y_m,\rho/2)}u^{-2}_0(|\n u_0|^2-\beta) e^{-||f||_0}\dvolg\\
\end{split}
\end{equation}where $\beta$ is small enough such that $|\n u_0|^2(z)-\beta>0$ for $z\in B(x,\tilde\delta)$. By compactness, we have $\inf_{z\in B(x,\tilde\delta)}u^{-2}_0(|\n u_0|^2-\beta)>0$. Therefore,we can conclude, there exists a $\eta>0$ \begin{equation}\label{eq:bubble lower bound}
\begin{split}
          \int_{\p^*\O_{\e_m}\cap B(x,\rho)}u^{-2}_m|\n u_m|^2 e^{-f}\dvolg&\ge \inf_{z\in B(x,\tilde\delta)}[u^{-2}_0(|\n u_0|^2-\beta)](z)e^{-||f||_0}|\p^*\O_{\e_m}\cap B(y_m,\rho/2)|\\
          &\hspace{-4mm}\overset{\text{Lem \ref{lem:DensityEstimatesConvergence}}}{\ge}\inf_{z\in B(x,\tilde\delta)}[u^{-2}_0(|\n u_0|^2-\beta)](z)e^{-||f||_0}C_d\rho^{n-1}\\
          &\ge \eta\\
          &>0
\end{split}
\end{equation}

Furthermore, $|\p^*\O_{\e_m}\cap\phi^{-1}([-q,q])|$ has an upper bound independent of $m$. By the minimizing property, we have $\mathcal{P}_{\e_m}(\O_{\e_m})\le \mathcal{P}_{\e_m}(\O_0)$. Furthermore, $\O_{\e_m}\setminus \O_0\subset\phi^{-1}(0,\infty)$ and by construction of $h_{\e_m}$ in Lemma \ref{lemma:mu_bubble_function}, it is non-positive in $\O_{\e_m}\setminus \O_0$,  which obtains \begin{equation}
    \int_{\p^*\O_{\e_m}}u^\gamma_m e^{-f}\le \int_{\p^*\Omega_0}u^\gamma_m e^{-f}+\int(\chi_{\O_{\e_m}}-\chi_{\O_0})
h_{\e_m}u^\gamma_me^{-\left(\frac{n+3}{n+1}\right)f}\le  \int_{\p^*\Omega_0}u^\gamma_m e^{-f}\end{equation}Recall that $\p^*\Omega_0$ is compact, and by (\ref{eq:choice_a_1}), the R.H.S is bounded independent of $m$. Therefore, we have \begin{equation}\label{eq: bubble -q,q estimate}
\begin{split}
         |\p^*\O_{\e_m}\cap \phi^{-1}([-q,q])|\dfrac{1}{2^\gamma}e^{-||f||_0}\inf_{z\in  \phi^{-1}([-q,q]}u^\gamma(z) &\le\int_{\p^*\O_{\e_m}}u^\gamma_m e^{-f}\\
         &\le\int_{\p^*\Omega_0}u^\gamma_m e^{-f}\\
         &\le \int_{\p^*\Omega_0}u^\gamma e^{||f||_0}\\
         &\le e^{||f||_0}|\p^*\O_0|\sup_{z\in \p^*\Omega_0}u^\gamma(z) 
\end{split}
\end{equation}Therefore, there exists some $\tilde C=C(\gamma,||f||_0,|\p^*\Omega|,\inf_{z\in  \phi^{-1}([-q,q]}u^\gamma(z),\sup_{z\in  \phi^{-1}([-q,q]}u^\gamma(z),q)>0$ such that \begin{equation}
    |\p^*\O_{\e_m}\cap \phi^{-1}([-q,q])|\le \tilde C
\end{equation}Similarly, $ |\p^*\O_{\e_m}\cap \phi^{-1}([-q,q])|$ has a uniform upper bound independent of $m$ and \begin{equation}
     |\p^*\O_{\e_m}\cap B(x,r_m)|\le \tilde C
\end{equation}
Recall $(\ref{eq: 2nd variation final form})$, we have \begin{equation}
\begin{split}
        0&\le -c_0\int_{\p^*\O}(u^{-2}|\np u|^2)e^{-f}\dvolg+ Ce^{\frac{n+3}{n-1}||f||_{0}}c_0^{-1}\e\big|\p^*\Omega_\e\cap\phi^{-1}([-q,q])\big|\\
        &\quad+ \delta e^{||f||_{0}}|\p^*\O_\e\cap B(x,r)|
\end{split}
\end{equation}
Combine the estimate above $(\ref{eq:bubble lower bound}),(\ref{eq: bubble -q,q estimate})$, we have \begin{equation}
    \begin{split}
        0<\eta&\le\liminf_{m\to\infty} c_0  \int_{\p^*\O_{\e_m}\cap B(x,\rho)}u^{-2}_m|\n u_m|^2 e^{-f}\dvolg\\
        &\le\liminf_{m\to\infty} Ce^{\frac{n+3}{n-1}||f||_{0}}c_0^{-1}\e_m\big|\p^*\Omega_{\e_m}\cap\phi^{-1}([-q,q])\big|+\liminf_{m\to\infty}\delta_m e^{||f||_{0}}|\p^*\O_{\e_m}\cap B(x,r_m)|\\
        &\le \liminf_{m\to\infty} \tilde C \e_m+\liminf_{m\to\infty}\tilde C \delta_m\\
        &=0
    \end{split}
\end{equation}Providing a contradiction.

For $n= 3$, the range of $\gamma$ is exactly what we want. For $n\neq3$ and $n\ge 2$, the above estimate is true for all $0<c_0<\tilde{c}(N,n)$, we can take the limit $c_0\to 0^+$. Therefore, the range of $\gamma$ in $(\ref{eq:final range of gamma for n neq 3})$ improve to \begin{equation}
        0<\gamma<c_N^{-1}
\end{equation}Therefore, we can conclude that for $n\ge 2$, the range of $\gamma$ is 
 \begin{equation}
    0<\gamma<\left(\frac{1}{(n-1)\left(1 + \frac{n-1}{N}\right)} + \frac{n-1}{4}\right)^{-1}
\end{equation}This prove the theorem.
\end{proof}
\section{Alternative proof for spectral splitting theorem in classical Ricci tensor}\label{sec:Alt proof}
Recall the spectral splitting theorem in \cite{spectral_splitting},\cite{spectral_splitting_2}.\begin{theorem}
    Let $n\geq 2$, $\gamma<\frac{4}{n-1}$, and let $(M^n,g)$ be an $n$-dimensional smooth complete noncompact Riemannian manifold without boundary. Assume that $M$ has at least two ends and satisfies
    \begin{equation}
        \lambda_1(-\gamma\Delta+\Ric)\geq0.
    \end{equation}Here $\Ric:M\to\R$ by $$\mathrm{Ric}(x):=\inf_{\substack{v\in T_xM\\ g(v,v)=1}} \mathrm{Ric}_x(v,v).$$
    Then $\Ric\geq0$ on $M$. In particular, the manifold splits isometrically as $(M,g)\cong(\R\times N,\mathrm{d}t^2+g_N)$ for some compact manifold $N$ with nonnegative Ricci curvature.
\end{theorem}
\begin{proof}
    By the definition of $N$-Bakry \'Emery Ricci tensor \ref{def:N-BE-Ricci} and $f$-Laplacian, for $f$ is a constant function. It conincide to the classical Ricci tensor and Laplacian for any $N\in \R$\begin{equation}
        \begin{split}
             \Ric^N_f&:=\Ric+\mathrm{Hess}(f)-\dfrac{1}{N}df\otimes df=\Ric\\
             \Delta_f&:=\Delta-\<\n f,\n\cdot\>=\Delta
        \end{split}
    \end{equation}By Theorem \ref{thm:main_thm}, we have for $n\ge 2$ and if \begin{equation}
        \gamma< \left(\frac{1}{(n-1)\left(1 + \frac{n-1}{N}\right)} + \frac{n-1}{4}\right)^{-1}\tab\text{and}\tab \lambda_1(-\gamma\Delta_f-\Ric^N_f)\ge 0
    \end{equation}Then $\Ric^N_f\ge 0$ on $M$.\\
    Pick $f\equiv 1$, we have for any $N\in(0,\infty)$ if \begin{equation}
        \gamma< \left(\frac{1}{(n-1)\left(1 + \frac{n-1}{N}\right)} + \frac{n-1}{4}\right)^{-1}\tab\text{and}\tab \lambda_1(-\gamma\Delta-\Ric)\ge 0
    \end{equation}Then $\Ric\ge 0$ on $M$. Since the range is for any $N\in(0,\infty)$, we can take $N\to 0^+$ and obtain $\gamma<\frac{4}{n-1}$. This recover the sharp constant $\gamma$ in the spectral splitting theorem  \cite{spectral_splitting}\cite{spectral_splitting_2}.
\end{proof}
\bibliographystyle{alpha}
\bibliography{splitting.bib}
\end{document}